\documentclass[10pt]{article}
\usepackage[english]{babel}
\usepackage{latexsym}
\usepackage{amsmath}
\usepackage{amssymb}
\usepackage[latin1]{inputenc}
\usepackage{enumerate}
\usepackage{graphicx}

\newcommand{\EF}{\operatorname{EF}}
\renewcommand{\k}{\kappa}
\newcommand{\ran}{\operatorname{ran}}
\renewcommand{\o}{\omega}
\newcommand{\A}{\mathcal{A}}
\newcommand{\B}{\mathcal{B}}
\renewcommand{\a}{\alpha}
\renewcommand{\Diamond}{\diamondsuit}
\renewcommand{\leq}{\leqslant}
\renewcommand{\ge}{\geqslant}
\newcommand{\rest}{\restriction}

\usepackage[sc]{mathpazo}

\usepackage{bbm}
\usepackage{mathtools}
\usepackage[amsthm,amsmath,thmmarks,thref]{ntheorem}
\theoremseparator{.}

\theoremstyle{definition}
\newtheorem{definition}{Definition}[section]
\theoremstyle{plain}
\newtheorem*{thm*}{Theorem}
\newtheorem{theorem}[definition]{Theorem}
\newtheorem{lemma}[definition]{Lemma}

\newtheorem{claim}[definition]{Claim}
\newtheorem{corollary}[definition]{Corollary}

\newtheorem{Question}[definition]{Question}
\theoremstyle{remark}
\newtheorem{Remark}{Remark}

\title{A Borel-reducibility Counterpart of Shelah's Main Gap Theorem}
\author{Tapani Hyttinen, Vadim Kulikov, Miguel Moreno\\ University of Helsinki}
\date{}
\begin{document}
\maketitle
\begin{abstract}
  We study the Borel-reducibility of isomorphism relations of
complete first order theories and show the consistency of the
following: For all such theories T and T', if T is classifiable
and T' is not, then the isomorphism of models of T' is strictly above
the isomorphism of models of T with respect to Borel-reducibility.  In
fact, we can also ensure that a range of equivalence relations modulo
various non-stationary ideals are strictly between those isomorphism
relations.  The isomorphism relations are considered on models of some
fixed uncountable cardinality obeying certain restrictions.
  
\end{abstract}

\section{Introduction}

Throughout this article we assume that $\kappa$ is an uncountable
cardinal that satisfies $\kappa^{<\kappa}=\kappa$. The generalized
Baire space is the set $\kappa^\kappa$ with the bounded topology.  For
every $\zeta\in \kappa^{<\kappa}$, the set 
$$[\zeta]=\{\eta\in \kappa^\kappa \mid \zeta\subset \eta\}$$ 
is a basic open set. The open sets are of the form $\bigcup X$ where
$X$ is a collection of basic open sets. The collection of
$\kappa$-Borel subsets of $\kappa^\kappa$ is the smallest set which
contains the basic open sets and is closed under unions and
intersections, both of length $\kappa$.  A $\k$-Borel set is any
element of this collection.  We usually omit the prefix ``$\kappa-$''.
In \cite{Vau} Vought studied this topology in the case
$\kappa=\omega_1$ assuming CH and proved the following:
\begin{thm*}
  A set $B\subset \omega_1^{\omega_1}$ is Borel and closed under
  permutations if and only if there is a sentence $\varphi$ in
  $L_{\omega_1^+\omega_1}$ such that
  $B=\{\eta\mid \mathcal{A}_\eta\models\varphi\}$.
\end{thm*}

This result was generalized in \cite{FHK13} to arbitrary $\kappa$
that satisfies $\kappa^{<\kappa}=\kappa$. Mekler and
V\"{a}\"{a}n\"{a}nen continued the study of this topology in~\cite{MV}.

We will work with the subspace $2^\kappa$ with the relative subspace
topology.  A function $f\colon 2^\kappa\rightarrow 2^\kappa$ is \emph{Borel}, 
if for every open set $A\subseteq 2^\kappa$ the inverse image
$f^{-1}[A]$ is a Borel subset of $2^\kappa$. Let $E_1$ and $E_2$ be
equivalence relations on $2^\kappa$. We say that $E_1$ is 
\emph{Borel reducible} to $E_2$, if there is a Borel function $f\colon
2^\kappa\rightarrow 2^\kappa$ that satisfies $(x,y)\in E_1\Leftrightarrow
(f(x),f(y))\in E_2$.  We call $f$ a \emph{reduction} of $E_1$ to
$E_2$. This is denoted by $E_1\le_B E_2$ and if $f$ is continuous,
then we say that $E_1$ is \emph{continuously reducible} to $E_2$ and
this is denoted by $E_1\le_c E_2$.

The following is a standard way to code structures with domain $\kappa$
with elements of $2^\kappa$.  To define it, fix a countable relational
vocabulary $\mathcal{L}=\{P_n\mid n<\omega\}$. 

\begin{definition}
  Fix a bijection $\pi\colon \kappa^{<\omega}\to \kappa$. For every
  $\eta\in 2^\kappa$ define the $\mathcal{L}$-structure
  $\mathcal{A}_\eta$ with domain $\kappa$ as follows:
  For every relation $P_m$ with arity $n$, every tuple
  $(a_1,a_2,\ldots , a_n)$ in $\kappa^n$ satisfies 
  $$(a_1,a_2,\ldots , a_n)\in P_m^{\mathcal{A}_\eta}\Longleftrightarrow \eta(\pi(m,a_1,a_2,\ldots,a_n))=1.$$
\end{definition}

Note that for every $\mathcal{L}$-structure $\mathcal{A}$ there exists
$\eta\in 2^\kappa$ with $\mathcal{A}=\mathcal{A}_\eta$.  For club many
$\alpha<\kappa$ we can also code the $\mathcal{L}$-structures with
domain $\alpha$:

\begin{definition}
  Denote by $C_\pi$ the club
  $\{\alpha<\kappa\mid \pi[\alpha^{<\omega}]\subseteq \alpha\}$. For every
  $\eta\in 2^\kappa$ and every $\alpha\in C_\pi$ define the structure
  $\mathcal{A}_{\eta\restriction \alpha}$ with domain $\alpha$ as
  follows:
  For every relation $P_m$ with arity $n$, every tuple
  $(a_1,a_2,\ldots , a_n)$ in $\alpha^n$ satisfies $$(a_1,a_2,\ldots ,
  a_n)\in P_m^{\mathcal{A}_{\eta\restriction \alpha}}\Longleftrightarrow
  \eta\restriction_\alpha (\pi(m,a_1,a_2,\ldots,a_n))=1.$$
\end{definition}

For every $\alpha\in C_\pi$ and every $X\subseteq \alpha$ we will
denote the structure $\mathcal{A}_{F}$ by $\mathcal{A}_X$, where $F$
is the characteristic function of $X$.  We will work with two
equivalence relations on $2^\kappa$: the isomorphism relation and the equivalence
 modulo the non-stationary ideal.

\begin{definition}[The isomorphism relation] Assume $T$ is a complete first order
  theory in a countable vocabulary. We define $\cong^\kappa_T$ as the
  relation $$\{(\eta,\xi)\in 2^\kappa\times
  2^\kappa\mid (\mathcal{A}_\eta\models T, \mathcal{A}_\xi\models T,
  \mathcal{A}_\eta\cong \mathcal{A}_\xi)\text{ or }
  (\mathcal{A}_\eta\not\models T, \mathcal{A}_\xi\not\models T)\}.$$
\end{definition}
We will omit the superscript ``$\kappa$'' in $\cong^\kappa_T$ when it
is clear from the context. For every first order theory $T$ in a countable
vocabulary there is an isomorphism relation associated with $T$,
$\cong^\kappa_T$. For every stationary $X\subset \kappa$, we define
an equivalence relation modulo the non-stationary ideal associated
with $X$:

\begin{definition} For every $X\subset \kappa$ stationary, we define
  $E_X$ as the relation $$E_X=\{(\eta,\xi)\in 2^\k\times 2^\k\mid
  (\eta^{-1}[1]\triangle\xi^{-1}[1])\cap X \text{ is not
    stationary}\}$$ where $\triangle$ denotes the symmetric
  difference.
\end{definition}

For every regular cardinal $\mu<\kappa$ denote $\{\alpha<\kappa \mid
cf(\alpha)=\mu\}$ by $S_\mu^\kappa$.  A set $C$ is $\mu$-\emph{club}
if it is ubounded and closed under $\mu$-limits, i.e. if
$S^\k_\mu\setminus C$ is non-stationary. 
Accordingly, we will denote the equivalence relation $E_{X}$ for
$X=S^\kappa_\mu$ by $E^2_{\mu\text{-club}}$. Note that $(f,g)\in
E^2_{\mu\text{-club}}$ if and only if the set $\{\alpha<\k\mid f(\a)=g(\a)\}$
contains a $\mu$-club.


\section{Reduction to $E_X$}

Classifiable theories (superstable with NOTOP and NDOP) have a close
connection to the Ehrenfeucht-Fra\"{i}ss\'e games (EF-games for short). We will use them
to study the reducibility of the
isomorphism relation of classifiable theories. The following
definition is from \cite[Def 2.3]{HM15}:

\begin{definition}[The Ehrenfeucht-Fra\"{i}ss\'e game] 
  Fix an enumeration $\{X_\gamma\}_{\gamma<\kappa}$ of the elements of
  $\mathcal{P}_\kappa(\kappa)$ and an enumeration
  $\{f_\gamma\}_{\gamma<\kappa}$ of all the functions with both the domain and range
  in $\mathcal{P}_\kappa(\kappa)$. 
  For every $\alpha\leq\kappa$ the game $\EF^\alpha_\omega(\mathcal{A}\restriction_\alpha,\mathcal{B}\restriction_\alpha)$
on the restrictions 
$\mathcal{A}\rest\a$ and $\mathcal{B}\rest\a$
of the structures $\mathcal{A}$ and $\mathcal{B}$ with domain $\kappa$ is
defined as follows: In the $n$-th move, first ${\bf I}$ chooses an
ordinal $\beta_n<\alpha$ such that $X_{\beta_n}\subset \alpha$ and
$X_{\beta_{n-1}}\subseteq X_{\beta_n}$. Then ${\bf II}$ chooses an
ordinal $\theta_n<\alpha$ such that
$dom(f_{\theta_n}),\ran(f_{\theta_n})\subset \alpha$,
$X_{\beta_n}\subseteq dom(f_{\theta_n})\cap \ran(f_{\theta_n})$ and
$f_{\theta_{n-1}}\subseteq f_{\theta_n}$ (if $n=0$ then
$X_{\beta_{n-1}}=\emptyset$ and $f_{\theta_{n-1}}=\emptyset$). The
game ends after $\omega$ moves. Player ${\bf II}$ wins if
$\bigcup_{i<\omega}f_{\theta_i}\colon A\restriction_\alpha\rightarrow
B\restriction_\alpha$ is a partial isomorphism. Otherwise player
${\bf I}$ wins. If $\a=\k$ then this is the same as the
standard $\EF$-game which is usually denoted by $\EF^\k_\o$.
 
  When a player $P$ has a winning strategy in a game $G$, 
  we denote it by~$P\uparrow G$.
\end{definition}

The following lemma is proved in \cite[Lemma 2.4]{HM15} and is used in
the main result of this section which in turn is central to the main
theorem of this paper.
\begin{lemma}\label{lem:EFClub}
  If $\mathcal{A}$ and $\mathcal{B}$ are structures with domain $\kappa$, then
\begin{itemize}
\item ${\bf II}\uparrow\EF^\kappa_\omega(\mathcal{A},\mathcal{B})\Longleftrightarrow {\bf II}\uparrow\EF^\a_\omega (\mathcal{A}\restriction_\alpha,\mathcal{B}\restriction_\alpha)$ for
  club-many $\alpha$,
\item ${\bf I}\uparrow\EF^\kappa_\omega (\mathcal{A},\mathcal{B})\Longleftrightarrow {\bf I}\uparrow\EF^\a_\omega (\mathcal{A}\restriction_\alpha,\mathcal{B}\restriction_\alpha)$ for club-many $\alpha$. \qed
\end{itemize}
\end{lemma}

\begin{Remark}
  In \cite[Lemma 2.7]{HM15} it was proved that there exists a club $C_{\EF}$
  of $\alpha$ such that the relation defined by the game
  $$\{(\A,\B)\mid {\bf II}\uparrow\EF^\a_\o(\A\rest_\a,\B\rest_\a)\}$$
  is an equivalence relation.
\end{Remark}

\begin{Remark}\label{rem:Shelah}
  Shelah proved in \cite{Sh}, that if $T$ is classifiable then every
  two models of $T$ that are $L_{\infty,\kappa}$-equivalent are
  isomorphic. On the other hand $L_{\infty,\kappa}$-equivalence is
  equivalent to $\EF_\omega^\kappa$-equivalence. So for every 
  two models $\mathcal{A}$ and $\mathcal{B}$ of $T$ we have 
  ${\bf II}\uparrow\EF^\kappa_\omega (\mathcal{A},\mathcal{B})\Longleftrightarrow
  \mathcal{A}\cong \mathcal{B}$ and ${\bf I}\uparrow\EF^\kappa_\omega
  (\mathcal{A},\mathcal{B})\Longleftrightarrow \mathcal{A}\ncong
  \mathcal{B}$.
\end{Remark}

\begin{lemma}\label{cor:ClassCub}
  Assume $T$ is a classifiable theory and $\mu<\kappa$ is a regular
  cardinal. If $\Diamond_\kappa (X)$ holds then $\cong^\k_T$
  is continuously reducible to $E_X$.
\end{lemma}
\begin{proof}
  Let $\{S_\alpha\mid \alpha\in X\}$ be a sequence testifying
  $\Diamond_\kappa (X)$ and define the function
  $\mathcal{F}\colon 2^\kappa\to 2^\kappa$ by 
  $$\mathcal{F}(\eta)(\alpha)=\begin{cases} 1 &\mbox{if } \alpha\in X\cap C_\pi\cap C_{EF},\ {\bf II}\uparrow \textit{EF}^\kappa_\omega (\mathcal{A}_\eta\restriction_\alpha, \mathcal{A}_{S_\alpha}) \mbox{ and } \mathcal{A}_\eta\restriction_\alpha\models T\\
    0 & \mbox{otherwise. } \end{cases}
  $$
  Let us show that $\mathcal{F}$ is a reduction of $\cong_T$ to
  $E_X$, i.e. for every $\eta,\xi\in 2^\kappa$,
  $(\eta,\xi)\in\ \cong_T$ if and only if
  $(\mathcal{F}(\eta),\mathcal{F}(\xi))\in
  E_X$. Notice that when $\alpha\in C_\pi$, the
  structure $\mathcal{A}_{\eta\restriction\alpha}$ is defined and equals
  $\mathcal{A}_\eta\restriction_\alpha$.

  Consider first the direction from left to right.  Suppose first that
  $\mathcal{A}_\eta$ and $\mathcal{A}_\xi$ are models of $T$ and
  $\mathcal{A}_\eta\cong \mathcal{A}_\xi$.  Since $\mathcal{A}_\eta\cong
  \mathcal{A}_\xi$, we have ${\bf II}\uparrow$ EF$^\kappa_\omega
  (\mathcal{A}_\eta,\mathcal{A}_\xi)$. By Lemma \ref{lem:EFClub} there is a club $C$
  such that ${\bf II}\uparrow\EF^\a_\omega
  (\mathcal{A}_\eta\restriction_\alpha,\mathcal{A}_\xi\restriction_\alpha)$
  for every $\alpha$ in $C$. Since the set
  $\{\alpha<\kappa\mid \mathcal{A}_\eta\restriction_\alpha\models T,
  \mathcal{A}_\xi\restriction_\alpha\models T\}$ contains a club, we can
  assume that every $\alpha\in C$ satisfies
  $\mathcal{A}_\eta\restriction_\alpha\models T$ and
  $\mathcal{A}_\xi\restriction_\alpha\models T$.  If $\alpha\in C$ is
  such that $\mathcal{F}(\eta)(\alpha)=1$, then 
  ${\bf II}\uparrow\EF^\a_\omega (\mathcal{A}_\eta\restriction_\alpha,
  \mathcal{A}_{S_\alpha})$. Since ${\bf II}\uparrow \EF^\a_\omega
  (\mathcal{A}_\eta\restriction_\alpha,\mathcal{A}_\xi\restriction_\alpha)$
  and $\alpha\in C_{EF}$, we can conclude that ${\bf II}\uparrow
  \EF^\a_\omega
  (\mathcal{A}_\xi\restriction_\alpha,\mathcal{A}_{S_\alpha})$. Therefore
  for every $\alpha\in C$, $\mathcal{F}(\eta)(\alpha)=1$ implies
  $\mathcal{F}(\xi)(\alpha)=1$. Using the same argument it can be shown
  that for every $\alpha\in C$, $\mathcal{F}(\xi)(\alpha)=1$ implies
  $\mathcal{F}(\eta)(\alpha)=1$. Therefore $\mathcal{F}(\eta)$ and
  $\mathcal{F}(\xi)$ coincide in a
  club and $(\mathcal{F}(\eta),\mathcal{F}(\xi))\in
  E_X$.
  
  Let us now look at the case where $(\eta,\xi)\in \ \cong_T$ and
  $\mathcal{A}_\eta$ is not a model of $T$ (the case
  $T\not\models\A_\xi$ follows by symmetry).  By the definition of
  $\cong_T$ we know that $\mathcal{A}_\xi$ is not a model of $T$ either,
  so there is $\varphi\in T$ such that $\mathcal{A}_\eta\models \neg
  \varphi$ and $\mathcal{A}_\xi\models \neg\varphi$. Further, there is a club $C$ such
  that for every $\alpha\in C$ we have
  $\mathcal{A}_\eta\restriction_\alpha\models \neg \varphi$ and
  $\mathcal{A}_\xi\restriction_\alpha\models \neg\varphi$. We conclude that for every
  $\alpha\in C$ we have that $\mathcal{A}_\eta\restriction_\alpha$
  and $\mathcal{A}_\xi\restriction_\alpha$ are not models of $T$, and
  $\mathcal{F}(\eta)(\alpha)=\mathcal{F}(\xi)(\alpha)=0$, so
  $(\mathcal{F}(\eta),\mathcal{F}(\xi))\in E_X$.
  
  Let us now look at the direction from right to left. 
  Suppose first that
  $\mathcal{A}_\eta$ and $\mathcal{A}_\xi$ are models of $T$, and
  $\mathcal{A}_\eta\not\cong \mathcal{A}_\xi$.
  
  By Remark~\ref{rem:Shelah}, we know that ${\bf I}\uparrow\EF^\kappa_\omega (\mathcal{A}_\eta,\mathcal{A}_\xi)$.
  By Lemma \ref{lem:EFClub} there is a club $C$ of $\a$ with
  $${\bf I}\uparrow\EF^\a_\omega (\mathcal{A}_\eta\restriction_\alpha, \mathcal{A}_\xi\restriction_\alpha),$$
  $\mathcal{A}_\xi\restriction_\alpha\models T$ and
  $\mathcal{A}_\eta\restriction_\alpha\models T$.

  Since $\{\alpha\in X\mid \eta\cap \alpha=S_\alpha\}$ is
  stationary by the definition of $\Diamond_\k(X)$, also the set
  $$\{\alpha\in X\mid \eta\cap \alpha=S_\alpha\}\cap C_\pi\cap C_{EF}$$
  is stationary and every
  $\alpha$ in this set satisfies ${\bf II}\uparrow
  \text{EF}^\kappa_\omega (\mathcal{A}_\eta\restriction_\alpha,
  \mathcal{A}_{S_\alpha})$.  Therefore 
  $$C\cap \{\alpha\in X\mid \eta\cap \alpha=S_\alpha\}\cap C_\pi\cap C_{EF}$$
  is stationary and a subset of $\mathcal{F}(\eta)^{-1}\{1\}\ \triangle\ \mathcal{F}(\xi)^{-1}\{1\}$, where $\triangle$ denotes the symmetric
  difference. We conclude that $(\mathcal{F}(\eta),\mathcal{F}(\xi))\notin E_X$.

  Let us finally assume that $(\eta,\xi)\notin\ \cong_T$ and $\mathcal{A}_\eta\not\models T$ 
  (the case $\A_\xi\not\models T$ follows by symmetry).
  Assume towards a contradiction that 
  $(\mathcal{F}(\eta),\mathcal{F}(\xi))\in E^2_{\mu\text{-club}}$. 
  Let $C$ be a club that testifies 
  $(\mathcal{F}(\eta),\mathcal{F}(\xi))\in E^2_{\mu\text{-club}}$, i.e. $C\cap (\mathcal{F}(\eta)^{-1}[1]\triangle \mathcal{F}(\xi)^{-1}[1])\cap X=\emptyset$. 
  Since $\mathcal{A}_\eta\not\models T$, the set 
  $\{\alpha<\kappa\mid \mathcal{A}_\eta\restriction_\alpha\not\models T\}$ contains a club.
  Hence, we can assume that for every $\alpha\in C$, $\mathcal{A}_\eta\restriction_\alpha\not\models T$
  which implies that $\mathcal{F}(\eta)(\alpha)=0$ and $\mathcal{F}(\xi)(\alpha)=0$ for every $\alpha\in C$.
  
  By the definition of $\cong_T$, $\mathcal{A}_\eta\not\models T$
  implies $\mathcal{A}_\xi\models T$. Therefore the set
  $\{\alpha<\kappa\mid \mathcal{A}_\xi\restriction_\alpha\models T\}$
  contains a club. So there is a club $C'$ such that every $\alpha\in C'$ 
  satisfies $\mathcal{A}_\xi\restriction_\alpha\models T$ and
  $\mathcal{F}(\xi)(\alpha)=0$.
  Since $\{\alpha\in X\mid \xi\cap \alpha=S_\alpha\}$ is stationary, again by the definition of 
  $\Diamond_\k(X)$, also 
  $\{\alpha\in X\mid\eta\cap \alpha=S_\alpha\}\cap C_\pi\cap C_{EF}$ is stationary
  and every $\alpha$ in this set satisfies 
  ${\bf II}\uparrow \text{EF}^\kappa_\omega (\mathcal{A}_\eta\restriction_\alpha, \mathcal{A}_{S_\alpha})$. 
  Therefore, 
  $$C'\cap \{\alpha\in X\mid\xi\cap \alpha=S_\alpha\}\cap C_\pi\cap C_{EF}\neq \emptyset,$$
  a contradiction.

  To show that $\mathcal{F}$ is continuous, let
  $[\eta\restriction_\alpha]$ be a basic open set, $\xi\in
  \mathcal{F}^{-1}[[\eta\restriction_\alpha]]$. Then $\xi\in
  [\xi\restriction_\a]$ and $[\xi\restriction_\a]\subseteq
  \mathcal{F}^{-1}[[\eta\restriction_\alpha]]$. We conclude that
  $\mathcal{F}$ is continuous.
\end{proof}

To define the reduction $\mathcal{F}$ it is not enough to use the
isomorphism classes of the models $\mathcal{A}_{S_\alpha}$, as opposed
to the equivalence classes of the relation defined by the EF-game. It
is possible to construct two non-isomorphic models with domain
$\kappa$ such that their restrictions to any $\alpha<\kappa$ are
isomorphic. For example the models $\mathcal{M}=(\kappa,P)$ and
$\mathcal{N}=(\kappa,Q)$, with $\kappa=\lambda^+$,
$$P=\{\alpha<\kappa\mid\alpha=\beta+2n,\ n\in \mathbb{N}\text{ and }\beta\text{ a limit ordinal}\}$$
and 
$$Q=\{\alpha<\lambda\mid\alpha=\beta+2n,\ n\in \mathbb{N}\text{ and }\beta \text{ a limit ordinal}\}$$
are non-isomorphic but
$\mathcal{M}\restriction_\alpha\cong\mathcal{N}\restriction_\alpha$
holds for every $\alpha<\kappa$.

The Borel reducibility of the isomorphism relation of classifiable
theories was studied in \cite{FHK13} and one of the main results is the
following.

\begin{theorem}(\cite[Thm 77]{FHK13})
  If a first order theory $T$ is classifiable, then for all regular cardinals $\mu<\kappa$, $E^2_{\mu\text{-club}}\nleq_B\ \cong^\k_T$.
\end{theorem}

\begin{corollary}
  Assume that $\Diamond_\kappa (S_\mu^\kappa)$ holds for all regular
  $\mu<\kappa$. If a first order theory $T$ is classifiable, then for
  all regular cardinals $\mu<\kappa$ we have
  $\cong^\k_T\ \leq_c E^2_{\mu\text{-club}} $ and $E^2_{\mu\text{-club}}\nleq_B\  \cong^\k_T$.
\end{corollary}


\section{Non-classifiable Theories}

In \cite{FHK13} the reducibility to the
isomorphism of non-classifiable theories was studied.
In particular the following two theorems were proved there:

\begin{theorem}\label{thm:FHK79}(\cite[Thm 79]{FHK13})
Suppose that $\kappa=\lambda^+=2^\lambda$ and $\lambda^{<\lambda}=\lambda$.
\begin{enumerate}
\item If $T$ is unstable or superstable with OTOP, then $E^2_{\lambda\text{-club}}\leq_c\  \cong^\k_T$.
\item If $\lambda\ge 2^\omega$ and $T$ is superstable with DOP, then $E^2_{\lambda\text{-club}}\leq_c\ \cong^\k_T$. \qed
\end{enumerate}
\end{theorem}

\begin{theorem}\label{thm:FHK86} (\cite[Thm 86]{FHK13})
  Suppose that for all $\gamma<\kappa$, $\gamma^\omega<\kappa$ and $T$ is a
  stable unsuperstable theory. Then $E^2_{\omega\text{-club}}\leq_c\ \cong^\k_T$. \qed
\end{theorem}

Clearly from Theorems \ref{thm:FHK79} and \ref{thm:FHK86} and Corollary~\ref{cor:ClassCub} we obtain the following:

\begin{theorem}\label{thm:Main1} 
  Suppose that $\kappa=\lambda^+=2^\lambda$, $\lambda^{<\lambda}=\lambda$ and $\Diamond_\kappa (S_\lambda^\kappa)$ holds. 
\begin{enumerate}
\item If $T_1$ is classifiable and $T_2$ is unstable or superstable with OTOP, then $\cong^\k_{T_1}\ \leq_c\ \cong^\k_{T_2}$ and $\cong^\k_{T_2}\ \not\leq_B\ \cong^\k_{T_1}$.
\item If $\lambda\ge 2^\omega$, $T_1$ is classifiable and $T_2$ is superstable with DOP, then $\cong^\k_{T_1}\ \leq_c\ \cong^\k_{T_2}$ and $\cong^\k_{T_2}\ \not\leq_B\ \cong^\k_{T_1}$.
\end{enumerate}
\end{theorem}

\begin{theorem}\label{thm:Main2}
  Suppose that for all $\gamma<\kappa$, $\gamma^\omega<\kappa$ and
  $\Diamond_\kappa (S_\omega^\kappa)$ holds. If $T_1$ is
  classifiable and $T_2$ is stable unsuperstable, then
  $\cong^\k_{T_1}\ \leq_c\ \cong^\k_{T_2}$ and
  $\cong^\k_{T_2}\ \not\leq_B\ \cong^\k_{T_1}$.
\end{theorem}

\begin{corollary}
  Suppose $\kappa=\kappa^{<\kappa}=\lambda^+$ and
  $\lambda^\omega=\lambda$. If $T_1$ is classifiable and
  $T_2$ is stable unsuperstable, then $\cong^\k_{T_1}\ \leq_c\ \cong^\k_{T_2}$ and $\cong^\k_{T_2}\ \not\leq_B\ \cong^\k_{T_1}$.
\end{corollary}
\begin{proof}
  In \cite{Sh10} Shelah proved that if $\kappa=\lambda^+=2^\lambda$
  and $S$ is a stationary subset of $\{\alpha<\kappa\mid cf(\alpha)\neq cf(\lambda)\}$, 
  then $\Diamond_\kappa(S)$ holds. Since
  $\lambda^\omega=\lambda$, we have $cf(\lambda)\neq \omega$ and
  $\Diamond_\kappa (S_\omega^\kappa)$ holds. On the other hand
  $\kappa=\lambda^+$ and $\lambda^{\omega}=\lambda$ implies
  $\gamma^\omega<\kappa$ for all $\gamma<\kappa$. By Theorem~\ref{thm:Main2} we
  conclude that if $T_1$ is a classifiable theory and $T_2$ is a
  stable unsuperstable theory, then $\cong_{T_1}\ \leq_c\ \cong_{T_2}$
  and $\cong_{T_2}\ \not\leq_B\ \cong_{T_1}$.
\end{proof}

\begin{theorem}
  Let $H(\kappa)$ be the following property: If $T$ is classifiable and $T'$ not,
  then $\cong^\k_T\ \leq_c\ \cong^\k_{T'}$ and $\cong^\k_{T'}\ \not\leq_B\ \cong^\k_T$.
  Suppose that $\kappa=\kappa^{<\kappa}=\lambda^+$, $2^\lambda>2^\omega$ and $\lambda^{<\lambda}=\lambda$.
  \begin{enumerate}
  \item If $V=L$, then $H(\kappa)$ holds.
  \item There is a $\kappa$-closed forcing notion $\mathbb{P}$ with the $\kappa^+$-c.c.
    which forces $H(\kappa)$.
  \end{enumerate}
\end{theorem}
\begin{proof}
  \begin{enumerate}
  \item This follows from Theorems \ref{thm:Main1} and \ref{thm:Main2}.
  \item Let $\mathbb{P}$ be $\{f\colon X\rightarrow 2\mid X\subseteq\kappa,|X|<\kappa\}$ 
    with the order $p\leq q$ if $q\subset p$. It
    is known that $\mathbb{P}$ has the $\kappa^+$-cc \cite[Lemma IV.7.5]{KK}
    and is $\kappa$-closed \cite[Lemma IV.7.14]{KK}. It is
    also known that $\mathbb{P}$ preserves cofinalities, cardinalities
    and subsets of $\kappa$ of size less than $\kappa$ 
    \cite[Thm IV.7.9, Lemma IV.7.15]{KK}. Therefore, in $V[G]$, $\kappa$
    satisfies $\kappa=\kappa^{<\kappa}=\lambda^+=2^\lambda>2^\omega$
    and $\lambda^{<\lambda}=\lambda$. It is known that $\mathbb{P}$
    satisfies $\mathbb{1}\Vdash_\mathbb{P} \Diamond_k(S_\mu^\kappa)$
    for every regular cardinal $\mu<\kappa$. Therefore, by Theorems  \ref{thm:Main1} and \ref{thm:Main2}
    $H(\kappa)$ holds in~$V[G]$.
\end{enumerate}
\end{proof}

\begin{definition}
  \begin{enumerate}
  \item A tree $T$ is a \emph{$\k^{+},\k$-tree}
    if does not contain chains of length $\k$ and
    its cardinality is less than $\k^+$. 
    It is \emph{closed} if every chain has
    a unique supremum.
  \item A pair $(T,h)$ is a \emph{$Borel^{*}$-code}
    if $T$ is a closed $\k^{+},\k$-tree and $h$ is a function
    with domain $T$ such that if $x\in T$ is a leaf, then
    $h(x)$ is a basic open set and otherwise
    $h(x)\in\{\cup ,\cap\}$.
  \item For an element $\eta\in 2^\k$ and a $Borel^{*}$-code
    $(T,h)$, the \emph{$Borel^{*}$-game} $B^{*}(T,h,\eta)$ is
    played as follows.  There are two players, ${\bf I}$ and
    ${\bf II}$. The game starts from the root of $T$. At each move, if
    the game is at node $x\in T$ and $h(x)=\cap$, then ${\bf I}$
    chooses an immediate successor $y$ of $x$ and the game continues
    from this $y$. If $h(x)=\cup$, then ${\bf II}$ makes the choice.  At
    limits the game continues from the (unique) supremum of the
    previous moves by Player {\bf I}. 
    Finally, if $h(x)$ is a basic open set, then the
    game ends, and ${\bf II}$ wins if and only if $\eta\in h(x)$.
  \item A set $X\subseteq 2^\k$ is a \emph{$Borel^{*}$-set}
    if there is a $Borel^{*}$-code
    $(T,h)$ such that for all $\eta\in 2^\k$,
    $\eta\in X$ if and only if ${\bf II}$ has a winning strategy in the game
    $B^{*}(T,h,\eta)$.
  \end{enumerate}
\end{definition}

Notice that a strategy in a game $B^{*}(T,h,\eta)$ can be seen as a
function $\sigma: \k^{<\kappa}\rightarrow\k$, because every $\k^+\k$-tree
can be seen as a downward closed subtree of $\k^{<\k}$.

Suppose $T$ is a closed $\k^{+},\k$-tree and $h$ is a function
    with domain $T$ such that if $x\in T$ is a leaf, then
    $h(x)$ is a Borel set and otherwise
    $h(x)\in\{\cup ,\cap\}$. Then the set $$\{\eta\mid \textit{{\bf II} has a winning strategy in the game }B^*(T,h,\eta) \}$$ is a Borel$^*$-set. This can be seen by the Borel$^*$-code $(T',h')$, where $T'$ is $T$ concatenated in every leaf, $b$, by the tree $T_b$ and $h'$ is the union of the functions $h_b$, where $(T_b,h_b)$ is a Borel$^*$-code for the set $h(b)$.

\begin{theorem}
Suppose that $\kappa=\kappa^{<\kappa}=\lambda^+$, $2^\lambda>2^\omega$ and $\lambda^{<\lambda}=\lambda$. Then the following statements are consistent.
\begin{enumerate}
\item If $T_1$ is classifiable and $T_2$ is not, then there is an
  embedding of $(\mathcal{P}(\k), \subseteq)$ to
  $(B^*(T_1,T_2),\leq_B)$, where $B^*(T_1,T_2)$ is the set of all
  Borel$^*$-equivalence relations strictly between $\cong_{T_1}$ and
  $\cong_{T_2}$.
\item If $T_1$ is classifiable and $T_2$ is unstable or superstable,
  then $$\cong^\k_{T_1}\ \leq_c  E^2_{\lambda\text{-club}}\leq_c\ \cong^\k_{T_2}\land
  \cong^\k_{T_2}\ \not\leq_B E^2_{\lambda\text{-club}}\land
  E^2_{\lambda\text{-club}}\not\leq_B\ \cong^\k_{T_1}.$$
\end{enumerate}
\end{theorem}
\begin{proof}
We will start the proof with two claims.
\begin{claim}
If $\Diamond_\k(S)$ holds in $V$ and $\mathbb{Q}$ is $\k$-closed, then $\Diamond_\k(S)$ holds in every $\mathbb{Q}$-generic extension.
\end{claim}
\begin{proof}
Let us proceed by contradiction. Suppose $(S_\a)_{\a\in S}$ is a $\Diamond_\k(S)$-sequence in $V$ but not in $V[G]$, for some generic $G$. Fix the names $\check{S},\dot{C},\dot{X}\in V^\mathbb{Q}$ and $p\in G$, such that: 
$$p\Vdash (\dot{C}\subseteq \check{\k} \text{ is a club }\wedge \dot{X}\subseteq \check{\k}\wedge\forall \a\in\dot{C}[\check{S}_\a\neq \dot{X}\cap\a]).$$ Working in $V$, we choose by recursion $p_\a,\beta_\a,\theta_\a$,and $\delta_\a$ such that:
\begin{enumerate}
\item $p_\a\in \mathbb{Q}$, $p_0=p$ and $p_\a\ge p_\gamma$ if $\a\leq \gamma$.
\item $\beta_\a\leq \beta_\gamma$ if $\a\leq \gamma$.
\item $\beta_\a\leq\theta_\a,\delta_\a< \beta_{\a+1}$.
\item If $\gamma$ is a limit ordinal, then $\beta_\gamma=\delta_\gamma=\cup_{\a<\gamma} \beta_{\a}$.
\item $p_{\a+1}\Vdash (\check{\delta}_\a\in\check{C}\wedge \dot{X}\cap\check{\beta}_\a=\check{S}_{\theta_\a})$.
\end{enumerate}
We will show how to choose them such that 1-5 are satisfied.  First,
for the successor step assume that for some $\alpha<\k$ we have chosen
$p_{\a+1},\beta_\a,\theta_\a$ and $\delta_\a$. We choose any ordinal
satisfying 3 as $\beta_{\a+1}$. Since $p_{\a+1}\Vdash
(\dot{C}\subseteq \check{\k} \text{ is a club})$, there exists
$q\in \mathbb{Q}$ stronger than $p_{\a+1}$ and $\delta<\k$ such that
$q\Vdash( \check{\delta}\in\dot{C}\wedge\check{\beta}_\a
\leq\check{\delta})$. Now set $\delta_{\a+1}=\delta$. Since
$\mathbb{Q}$ is $\k$-closed, there exists $Y\in
\mathcal{P}(\beta_{\a+1})^V$ and $r\in \mathbb{Q}$ stronger than $q$
such that $r\Vdash \dot{X}\cap \check{\beta}_{\a+1}=\check{Y}$. By
$\Diamond_\k(S)$ in $V$, the set $\{\gamma<\k\mid Y=S_\gamma\}$ is
stationary, so we can choose the least ordinal
$\theta_{\a+1}\ge\beta_{\a+1}$ such that $r\Vdash \dot{X}\cap
\check{\beta}_{\a+1}=\check{S}_{\theta_{\a+1}}$. Clearly $r=p_{\a+2}$
satisfies 1 and~5.

For the limit step, assume that for some limit ordinal $\alpha<\k$ we have
chosen $p_{\gamma},\beta_\gamma,\theta_\gamma$ and $\delta_\gamma$ for
every $\gamma<\alpha$. Note that by 4 we know how to choose
$\beta_\a$ and $\delta_\a$. Since $\mathbb{Q}$ is $\k$-closed, there
exists $p_\a$ that satisfies 1. We choose $\theta_\a$ as in the
successor case with $q=p_\a$ and $p_{\a+1}$ as the condition
$r$ used to choose $\theta_\a$.

Define $A,B$ and $C_\delta$ by $B=\cup_{\a<\k}S_{\theta_\a}$,
$A=\{\a\in S\mid B\cap\a=S_\a\}$ and $C_\delta=\{\delta_\a\mid
\a\text{ is a limit ordinal}\}$. Note that $C_\delta$ is a club. By
$\Diamond_\k(S)$ in $V$, $A$ is stationary and $A\cap
C_\delta\neq\emptyset$. Let $\delta_\a\in A\cap C_\delta$. Then by 1, 2 and
5, for every $\gamma>\a$ we have $p_{\gamma+1}\Vdash
(\check{S}_{\theta_\a}=\check{S}_{\theta_\gamma}\cap\check{\beta}_\a)$. Therefore,
$S_{\theta_\a}=B\cap \beta_\a$ and $\delta_\a\in A\cap C_\delta$ and so by
4 we have $S_{\theta_\a}=B\cap \delta_\a=S_{\delta_\a}$. But now by 5 
we get
$p_{\a+1}\Vdash
(\check{\delta}_\a\in\check{C}\wedge
\dot{X}\cap\check{\delta}_\a=\check{S}_{\delta_\a})$ which is a contradiction.
\end{proof}

\begin{claim}
For all stationary 
$X\subseteq \kappa$, the relation $E_X$ is a Borel$^*$-set.
\end{claim}
\begin{proof}
The idea is to code the club-game into the Borel$^*$-game: in the club-game
the players pick ordinals one after another and if the limit is in a predefined
set $A$, then the second player wins. 
Define $T_X$ as the tree whose elements are all the
increasing elements of $\kappa^{\leq\lambda}$, ordered by
end-extension. For every element of $T_X$ that is not a leaf, define 
$$H_X(x)=\begin{cases} \cup &\mbox{if } x \text{ has an immediate predecessor } x^- \text{ and } H_X(x^-)=\cap\\ 
  \cap & \mbox{otherwise } \end{cases}
  $$
and for every leaf $b$ define $H_X(b)$ by:
$$(\eta,\xi)\in H_X(b)\Longleftrightarrow \text{ for every }\a\in \lim(\ran(b))\cap X(\eta(\a)=\xi(\a))$$ where $\a\in \lim(\ran(b))$ if $sup(\a\cap \ran(b))=\a$.

Let us assume there is a winning strategy $\sigma$ for Player ${\bf II}$ in the
game $B^*(T_X,H_X,(\eta,\xi))$ and let us conclude that $(\eta,\xi)\in
E_X$. Clearly
by the definition of $H_X$ we know that $\eta$ and $\xi$ coincide in the set
$B=\{\a<\k\mid \sigma[dom(\sigma)\cap\a^{<\lambda}]\subset\a^{<\lambda}\}\cap
X$. Since $\lambda^{<\lambda}=\lambda$, we know that $B'=\{\a<\k\mid
\sigma[dom(\sigma)\cap\a^{<\lambda}]\subset\a^{<\lambda}\}$ is closed and unbounded. Therefore, there exists a club that doesn't intersect
$(\eta^{-1}[1]\triangle\xi^{-1}[1])\cap X$.

For the other direction, assume that $(\eta^{-1}[1]\triangle\xi^{-1}[1])\cap
X$ is not stationary and denote by $C$ the club that does not
intersect $(\eta^{-1}[1]\triangle\xi^{-1}[1])\cap X$. The second player has
a winning strategy for the game $B^*(T_\lambda,H_X,(\eta,\xi))$: she
makes sure that, if $b$ is the leaf in which the game ends and $A\subset \ran(b)$ is such that $\sup(\cup A)\in X$, then
$\sup(\cup A)\in C$. This can be done by always choosing elements $f\in
\kappa^{<\lambda}$ such that $\sup(\ran(f))\in C$.
\end{proof}

Let $\mathbb{P}$ be $\{f\colon X\rightarrow 2\mid X\subseteq\kappa,|X|<\kappa\}$ with the order $p\leq q$ if $q\subset p$. It is known that in any $\mathbb{P}$-generic extension, $V[G]$, $\Diamond_\k(S)$ holds for every $S\in V$, $S$ a stationary subset of $\k$. 
\begin{enumerate}
\item In \cite[Thm 52]{FHK13} the following was proved under the assumption $\k=\lambda^+$ and GCH:

\textit{For every $\mu<\k$ there is a $\kappa$-closed forcing notion $\mathbb{Q}$ with the $\kappa^+$-c.c.
    which forces that there are stationary sets $K(A)\subsetneq S^\k_\mu$ for each $A\subsetneq \kappa$ such that $E_{K(A)}\not\leq_B E_{K(B)}$ if and only if $A\not\subset B$.}
    
In \cite[Thm 52]{FHK13} the proof starts by taking $(S_i)_{i<\k}$, $\k$ pairwise disjoint stationary subsets of $lim( S_\mu^\k)=\{\alpha\in S^\k_\mu\mid \a \text{ is a limit ordinal in }S_\mu^\k\}$, and defining $K(A)=\cup_{\a\in A}S_\a$. $\mathbb{Q}$ is an iterated forcing that satisfies:
For every name $\sigma$ of a function $f:2^\k\to2^\k$, exists $\beta<\k$ such that, $\mathbb{P}_\beta\Vdash ``\sigma \textit{ is not a reduction}''$. 

With a small modification on the iteration, it is possible to construct $\mathbb{Q}$ a $\kappa$-closed forcing with the $\kappa^+$-c.c. that forces 
\begin{itemize}
\item[$(*)$] For $\mu\in\{\o,\lambda\}$ and $A\subsetneq \kappa$, there are stationary sets
  $K(\mu,A)\subsetneq S^\k_\mu$ for which $E_{K(\mu,A)}\not\leq_B E_{K(\mu,B)}$ if and only if
  $A\not\subset B$.
\end{itemize}
Without loss of generality we may assume that GCH holds in $V$. Let $G$ be a
$\mathbb{P}*\mathbb{Q}$-generic. It is enough to prove that for every
$A\subsetneq \k$ in $V[G]$ the following holds:
\begin{enumerate}
\item If $T_2$ is unstable, or superstable with OTOP or with DOP, then $E_{K(\lambda,A)}\in B^*(T_1,T_2)$.
\item If $T_2$ is stable unsuperstable, then $E_{K(\omega,A)}\in B^*(T_1,T_2)$.
\end{enumerate} 
In both cases the proof is the same; we will only consider (a).

Working in $V[G]$, let $T_2$ be as in (a). Since $\mathbb{Q}$
is $\kappa$-closed, we have $V[G]\models \Diamond_\kappa(S)$ for every 
stationary $S\subset \k$, $S\in V$. Since $\mathbb{P}$ and
$\mathbb{Q}$ are $\kappa$-closed and have the $\kappa^+$-c.c., we have
$\kappa=\kappa^{<\kappa}=\lambda^+$, $2^\lambda>2^\omega$ and
$\lambda^{<\lambda}=\lambda$. By Lemma \ref{cor:ClassCub}, Theorems
3.1 and 3.4, we have that $\cong^\k_{T_1}\ \leq_c
E_{K(\lambda,A)}\leq_c\ \cong^\k_{T_2}$ holds for every $A\subsetneq
\k$. The argument in the proof of Theorem 2.4 can be used to prove
that $E_{K(\lambda,A)}\not\leq_B \ \cong^\k_{T_1}$ holds for every
$A\subsetneq \k$.

To show that $\cong^\k_{T_2}\ \not\leq_B E_{K(\lambda,A)}$ holds for
every $A\subsetneq \k$, assume towards a contradiction that there exists
$B\subsetneq \k$ such that $\cong^\k_{T_2}\ \leq_B
E_{K(\lambda,B)}$. But then $E_{K(\lambda,A)}\leq_B E_{K(\lambda,B)}$
holds for every $A\subsetneq \k$ and by $(*)$, $A\subsetneq B$ for
every $A\subsetneq \k$. So $B=\k$ which is a contradiction.

\item In \cite[Thm 3.1]{HK12} it is proved (under the assumptions
  $2^\k=\k^+$ and $\kappa=\kappa^{<\kappa}>\o$) that there is a
  generic extension in which $\cong^\k_{DLO}$ is not a
  Borel$^*$-set. The forcing is constructed using the following claim
  \cite[Claim 3.1.5]{HK12}:

  \textit{For each $(t,h)$ there exists a $\k^+$-c.c. $\k$-closed
    forcing $\mathbb{R}(t,h)$ such that in any
    $\mathbb{R}(t,h)$-generic extension $\cong^\k_{DLO}$ is not a
    Borel$^*$-set.}

  The forcing in \cite[Thm 3.1]{HK12} works for every theory $T$ that
  is unstable, or $T$ non-classifiable and superstable (not only
  $DLO$, see \cite{HK12} and \cite{HT91}). Therefore, this claim can
  be generalized to:

\textit{For each $(t,h)$ there exists a $\k^+$-c.c. $\k$-closed forcing $\mathbb{R}(t,h)$ such that in any $\mathbb{R}(t,h)$-generic extension, $\cong^\k_{T}$ is not a Borel$^*$-set, for all $T$ unstable, or $T$ non-classifiable and superstable.}

By iterating this forcing (as in \cite[Thm 3.1]{HK12}), we construct a $\kappa$-closed forcing $\mathbb{Q}$, $\kappa^+$-c.c. that forces \textit{$\cong^\k_{T}$ is not a Borel$^*$-set}, for all $T$ unstable, or $T$ non-classifiable and superstable. 

Wwithout loss of generality we may assume that $2^\k=\k^+$ holds in $V$. Let
$G$ be a $\mathbb{P}*\mathbb{Q}$-generic. Since $\mathbb{Q}$ is
$\kappa$-closed, $V[G]\models \Diamond_\kappa(S)$ for every 
stationary $S\subset \k$, $S\in V$. Since $\mathbb{P}$ and
$\mathbb{Q}$ are $\kappa$-closed and have the $\kappa^+$-c.c., we have
$\kappa=\kappa^{<\kappa}=\lambda^+$, $2^\lambda>2^\omega$ and
$\lambda^{<\lambda}=\lambda$. Working in $V[G]$, let $T_2$ be
unstable, or non-classifiable and superstable. By Lemma
\ref{cor:ClassCub}, Theorems \ref{thm:Main1} and \ref{thm:Main2}
we finally have that
$\cong^\k_{T_1}\ \leq_c E^2_{\lambda\text{-club}}\leq_c\ \cong^\k_{T_2}$
and $E^2_{\lambda\text{-club}}\not\leq_B \ \cong^\k_{T_1}$ holds.

Since $2^\k\times 2^\k$ is homeomorphic to $2^\k$, in order to finish
the proof, it is enough to show that if $f\colon 2^\k\to 2^\k$ is Borel,
then for all Borel$^*$-sets $A$, the set $f^{-1}[A]$ is a Borel$^*$. This
is because if $f$ were the reduction $\cong^\k_{T_2}\ \leq_B E^2_{\lambda\text{-club}}$,
we would have $(f\times f)^{-1}[E^2_{\lambda\text{-club}}]=\ \cong^\k_{T_2}$ and since
$E^2_{\lambda\text{-club}}$ is Borel$^*$, this would yield the latter Borel$^*$ as well.

\begin{claim}
  Assume $f\colon 2^\k\to 2^\k$ is a Borel function and $B\subset
  2^\k$ is Borel$^*$. Then $f^{-1}[B]$ is Borel$^*$.
\end{claim}
\begin{proof}
  Let $(T_B,H_B)$ be a Borel$^*$-code for $B$. Define the
  Borel$^*$-code $(T_A,H_A)$ by letting $T_B=T_A$ and $H_A(b)=f^{-1}[H_B(b)]$
  for every branch $b$ of $T_B$. Let $A$ be the Borel$^*$-set coded by
  $(T_A,H_A)$. Clearly, ${\bf II}\uparrow B^*(T_B,H_B,\eta)$ if and
  only if ${\bf II}\uparrow B^*(T_A,H_A,f^{-1}(\eta))$, so $f^{-1}[B]=A$.
\end{proof}
\end{enumerate}
\end{proof}

We end this paper with an open question:

\begin{Question}
  Is it provable in ZFC that $\cong^\k_T\ \lneq_B\ \cong^\k_{T'}$ 
  (note the strict inequality)  
  for all complete
  first-order theories $T$ and $T'$, $T$ classifiable and $T'$ not?
  How much can the cardinality assumptions on~$\k$ be relaxed?
\end{Question}


\providecommand{\bysame}{\leavevmode\hbox to3em{\hrulefill}\thinspace}
\providecommand{\MR}{\relax\ifhmode\unskip\space\fi MR }
\providecommand{\MRhref}[2]{%
  \href{http://www.ams.org/mathscinet-getitem?mr=#1}{#2}
}
\providecommand{\href}[2]{#2}

\end{document}